\documentclass[11pt]{article}
\usepackage{amsfonts,amsmath,latexsym,amsthm,amssymb,graphicx, color}

\usepackage{geometry}
\usepackage{float}

\usepackage[T1]{fontenc}

\usepackage[francais,english]{babel}
\usepackage[applemac]{inputenc}

\newtheorem{prop}{Proposition}[section]
\newtheorem{remark}[prop]{Remark}

\newtheorem{theorem}[prop]{Theorem}

\newtheorem{definition}[prop]{Definition}

\numberwithin{equation}{section}

\def\E{\mathbb{E}}
\def\P{\mathbb{P}}

\def\real{\mathbb{R}}

\def\1{\textbf{1}}

\def\F{\mathbb{F}}

\def \D{\mathbb{D}}
\def \E{\mathbb{E}}
\def \F{\mathbb{F}}

\def \P{\mathbb{P}}

\def \R{\mathbb{R}}

\def\Kc{{\cal K}}

\def\no{\noindent}

\def\={\;=\;}
\def\.{\;.}

\def \i{1\!\mbox{\rm I}}
\def\1{{\bf 1}}

 \def\normeL2#1{\left\|{#1}\right\|_{L^2}}

\def\E{\mathbb{E}}
\def\P{\mathbb{P}}
\def\h{\mathfrak{H}}

\def\cal#1{\mathcal{#1}}

\def\real{\mathbb{R}}
\def\D{\mathbb{D}}

\author{Peter Imkeller\footnote{Institut f\"ur Mathematik, Humboldt-Universit\"at zu Berlin, Unter den Linden 6, 10099 Berlin, Germany, \texttt{imkeller@math.hu-berlin.de}} \and Thibaut Mastrolia\footnote{Universit\'e Paris-Dauphine, CEREMADE UMR CNRS 7534, Place du mar\'echal De Lattre De Tassigny 75775 Paris cedex 16, \texttt{mastrolia@ceremade.dauphine.fr}} \and Dylan Possama\"i\footnote{Universit\'e Paris-Dauphine, CEREMADE UMR CNRS 7534, Place du mar\'echal De Lattre De Tassigny 75775 Paris cedex 16, \texttt{possamai@ceremade.dauphine.fr} }\and Anthony R\'eveillac\footnote{INSA de Toulouse, IMT UMR CNRS 5219, Universit\'e de Toulouse, 135 avenue de Rangueil, 31077 Toulouse Cedex 4, France, \texttt{anthony.reveillac@insa-toulouse.fr}}}

\title{A note on the Malliavin-Sobolev spaces}

\begin{document}

\maketitle

\abstract{\noindent In this paper, we provide a strong formulation of the stochastic G\^ateaux differentiability in order to study the sharpness of a new characterization, introduced in \cite{MPR}, of the Malliavin-Sobolev spaces. We also give a new internal structure of these spaces in the sense of sets inclusion.}

\vspace{1em} 

\vspace{1em}
\noindent
{\noindent \textbf{AMS 2010 subject classification:} Primary: 60H07; Secondary: 46B09.
\normalsize
}

\section{Introduction}

Consider the classical $W^{1,2}(\real^d)$ Sobolev space on $\real^d$ ($d\geq 1$) consisting of mappings $f$ which are square integrable and which admit a square integrable weak derivative. This space is known to be the completion of continuously differentiable functions with respect to the Sobolev seminorm. As proved by Nikodym, $f$ is an element of $W^{1,2}(\real^d)$ if and only if there exist versions of $f$ along particular directions (see \cite[(C) page 717]{sugita} for a precise statement) whose Radon-Nikodym derivatives are square integrable. In an infinite-dimensional framework, similar characterizations can also be obtained. For instance, consider the classical Wiener space $(\Omega,\mathcal F$) defined as the space of continuous functions $\Omega$, on $[0,T]$ (with $T$ a fixed positive real number) vanishing at $0$, and where $\mathcal F$ is the $\sigma$-algebra of Borel sets on $\Omega$, for the topology induced by the uniform norm. Let $\P_0$ be the Wiener measure on $\Omega$, that is the unique measure on $(\Omega,\mathcal F)$ which makes the coordinate process $W$ a standard Brownian motion. Let in addition $H$ be the Cameron-Martin space, that is the space of absolutely continuous functions on $[0,T]$ with square integrable (with respect to Lebesgue's measure on $[0,T]$) Radon-Nikodym derivative. In that context, one can define the two following spaces: 
\begin{itemize}
\item[(S)] the closure of so-called cylindrical random variables on $\Omega$ with respect to the Sobolev seminorm $\|\cdot\|_{1,2}$ (where the weak derivative is the Malliavin derivative).
\item[(K-S)] the set of random variables $Z$ on $\Omega$ which are \textit{ray absolutely continuous} (RAC) and which satisfies the \textit{stochastic G\^ateaux differentiability} property (SGD) \textit{i.e.} which admit a G\^ateaux derivative in any direction of the Cameron-Martin space $H$.  
\end{itemize}
More precisely, a random variable $Z$ on $(\Omega,\mathcal F,\mathbb P_0)$ satisfies Property (SGD), if for any element $h$ in $H$, $${\lim_{\varepsilon \to 0} \frac{Z\circ \tau_{\varepsilon h}-Z}{\varepsilon} \overset{\P_0}{=} \langle \widetilde{\mathcal{D} Z} ,h\rangle_H},$$ 
 where the limit is taken in probability, $\widetilde{\mathcal{D} Z}$ denotes the G\^ateaux derivative of $Z$ and where $\tau_{\varepsilon h}$ denotes the shift operator on the Wiener space.
The space (S) has been introduced by Shigekawa in \cite{shigekawa}, as a natural candidate to be a Sobolev space whose derivative is given by the Malliavin derivative (whose definition will be recalled in the next section), whereas the space (K-S) has been introduced by Kusuoka and Stroock in \cite{KusuokaStroock}. In comparison to the finite-dimensional case, Sugita proved in \cite{sugita} that these two spaces coincide. We will refer to them later on as $\D^{1,2}$. To be more precise, Characterization (K-S) (essentially) means that
$$ Z \in \D^{1,2} \ \Longleftrightarrow \ Z \textrm{ satisfies (RAC)} \textrm{ and (SGD)}.$$ 

\noindent
In practice Property (RAC) may not be easy to handle. Recently, it was shown in \cite[Theorem 4.1]{MPR} that the characterization (RAC) + (SGD) can actually be replaced, equivalently, by a stronger version of (SGD) that we name \textit{strong stochastically G\^ateaux differentiability}, or (SSGD) for short, which can be roughly stated as (a precise statement will be given in the next section):
for any element $h$ in $H$
$${\lim_{\varepsilon \to 0} \frac{Z\circ \tau_{\varepsilon h}-Z}{\varepsilon} \overset{L^q}{=} \langle \widetilde{\mathcal{D} Z},h\rangle_H},$$

where the limit is taken in $L^q$ spaces. According to \cite{MPR}, we obtain
$$ Z \in \D^{1,2}\ \Longleftrightarrow\ \exists q \in (1,2) \textrm{ s.t. } Z \textrm{ satisfies (SSGD$_2$(q))},$$
where a random variable $Z$ on $(\Omega,\mathcal F,\mathbb P_0)$ is said to satisfy (SSGD$_2$(q)) if ($Z$ and $\nabla Z$ are square integrable and) for any element $h$ in $H$
$${\lim_{\varepsilon \to 0} \frac{Z\circ \tau_{\varepsilon h}-Z}{\varepsilon} \overset{L^q}{=} \langle \nabla Z,h\rangle_H},$$ where $\nabla Z$ denotes the Malliavin derivative of $Z$.
In other words, the (RAC) property from the characterization of $\mathbb D^{1,2}$ can be dropped by strengthening property (SGD). One can also prove the following (see \cite[Lemma 4.2]{MPR}):   
\begin{equation} \label{charac:forall}Z \in \D^{1,2}\ \Longleftrightarrow \ \forall q \in (1,2), \; Z \textrm{ satisfies (SSGD$_2$(q))}.\end{equation}
In particular, one can wonder if
\begin{equation} \label{eq:pb} Z \in \D^{1,2}\ \Longleftrightarrow\  Z \textrm{ satisfies (SSGD$_2$(2))}?\end{equation}
In view of \eqref{charac:forall}, it is clear that $Z$ satisfies SSGD$_2$(2) implies that $Z$ is in $\D^{1,2}$. The question is then is it also a necessary condition? The aim of this article is to answer this question by the negative. Indeed we provide in Theorem \ref{th:main} a random variable $Z$ such that
$$ Z \in \D^{1,2} \text{ \underline{and} } Z \textrm{ does not satisfies (SSGD$_2$(2))}.$$
Incidentally, this question is related to proving that 
$$\mathbb{D}^{1,2+}:=\bigcup_{\varepsilon >0} \mathbb{D}^{1,2+\varepsilon} \subsetneq \mathbb{D}^{1,2},$$
which we address at the end of the paper (as Theorem \ref{thm:inclusion}).

\vspace{0.5em}
\noindent The rest of the paper is organized as follows. In Section \ref{preliminaries} we fix all the notations, give a rigorous definition of the strong stochastic G\^ateaux differentiability, and we highlight the wellposedness of this definition by giving some classical properties concerning the space of random variables which satisfy this (strong) differentiability property. Then, in Section \ref{mainresults} we answer the question given in \eqref{eq:pb} and we study more precisely the structure of the Malliavin-Sobolev spaces.
\section{Preliminaries}\label{preliminaries}

\subsection{Notations}

We fix throughout the paper a time horizon $T>0$. Let $\Omega:=C_0([0,T],\mathbb R)$ be the canonical Wiener space of continuous function $\omega$ from $[0,T]$ to $\mathbb R$ such that $\omega(0)=0$. $\mathcal F$ denotes the Borel $\sigma$-algebra on $\Omega$, for the uniform topology. Let $W:=(W_t)_{t\in [0,T]}$ be the canonical Wiener process, that is, for any $t$ in $[0,T]$, $W_t$ denotes the evaluation mapping: $W_t(\omega):=\omega_t$ for any element $\omega$ in $\Omega$. We set $\F^o:=(\mathcal F^o_t)_{t\in[0,T]}$ the natural filtration of $W$, and remind the reader that $\mathcal F=\mathcal F^o_T$. Under the Wiener measure $\P_0$, the process $W$ is a standard Brownian motion and we denote by $\F:=(\mathcal F_t)_{t\in[0,T]}$ the usual augmentation (which is right-continuous and complete) of $\F^o$ under $\P_0$. Unless otherwise stated, all the expectations considered in this paper will have to be understood as expectations under $\P_0$, and all notions of measurability for elements of $\Omega$ will be with respect to the filtration $\F$ or the $\sigma$-field $\mathcal F_T$. Unless otherwise stated, topological spaces are endowed with their Borel $\sigma$-algebra.

\vspace{0.5em}
\noindent We set
$$\h:=\left\{f:[0,T]\longrightarrow\mathcal \real, \text{ Borel-measurable, s.t. }\int_0^T |f(s)|^2 ds <+\infty \right\},$$
and we denote by $H$ the Cameron-Martin space defined as: 
$$H:=\left\{ h:[0,T] \longrightarrow \real, \; \exists \dot{h}\in\h, \; h(t)=\int_0^t \dot{h}(x)dx, \; \forall t\in [0,T]\right\},$$
which is an Hilbert space equipped with the inner product $\langle h_1,h_2 \rangle_{H}:=\int_0^T \dot{h_1}(t) \dot{h_2}(t) dt$, for any $(h_1, h_2)\in H\times H$, and with associated norm $\|h\|_H^2:=\langle h,h \rangle_{H}$.
For any $h$ in $H$, we will always denote by $\dot{h}$ a version of its Radon-Nykodym density with respect to the Lebesgue measure. 
Define next for any Hilbert space $\mathcal K$ and for any $p\geq 1$, $L^p(\Kc)$ as the set of all $\mathcal F_T$-measurable random variables $Z$ which are valued in an Hilbert space $\mathcal{K}$, and such that $\|Z\|_{L^p(\Kc)}^p<+\infty$, where
$$\|Z\|_{L^p(\Kc)}:=\left(\E\left[\|Z\|_{\mathcal{K}}^p\right]\right)^{1/p}.$$

\vspace{0.5em}
\no Let $\mathcal{S}$ be the set of polynomial cylindrical functionals, that is the set of random variables $Z$ of the form
\begin{equation}
\label{eq:cylindrical}
Z=f(W(h_1),\ldots,W(h_n)), \quad (h_1,\ldots,h_n) \in H^n, \; f \in \R^n[X], \text{ for some }n\geq 1,
\end{equation}
where $W(h):=\int_0^T \dot{h}_s dW_s$ for any $h$ in $H$. For any $Z$ in $\mathcal{S}$ of the form \eqref{eq:cylindrical}, the Malliavin derivative $\nabla Z$ of $Z$ is defined as the following $H$-valued random variable:
\begin{equation}
\label{eq:DF}
\nabla Z:=\sum_{i=1}^n f_{x_i}(W(h_1),\ldots,W(h_n)) h_i(\cdot),
\end{equation}
where $f_{x_i}:=\frac{df}{dx_i}$. Denote then by $\mathbb{D}^{1,p}$ the closure of $\mathcal{S}$ with respect to the Malliavin-Sobolev semi-norm $\|\cdot\|_{1,p}$, defined as:
$$ \|Z\|_{1,p}:=\left(\E\left[|Z|^p\right] + \E\left[\|\nabla Z\|_{H}^p\right]\right)^{1/p}. $$
One of the main tools that we will use throughout this paper is the shift operator along directions in the Cameron-Martin space. More precisely, for any $h\in H$, we define the following shift operator $\tau_{h}:\Omega\longrightarrow\Omega$ by
$$\tau_{h}(\omega):=\omega + h.$$
Note that the fact that $h$ belongs to $H$ ensures that $\tau_h$ is a measurable shift on the Wiener space. One of the main techniques when working with shifts on the path space is the famous Cameron-Martin formula.
\begin{prop}\label{prop.cam}
$($Cameron-Martin Formula, see e.g. \cite[Appendix B.1]{ustunelzakai}$)$ Let $Z$ be a $\mathcal F_T$-measurable random variable and let $h$ be in $H$. Then, when both sides are well-defined
$$ \E[Z\circ \tau_{h}] =\E\left[ Z \exp\left(\int_0^T \dot{h}(s) dW_s -\frac12 \int_0^T |\dot{h}(s)|^2 ds\right)\right]. $$
\end{prop}

\noindent Fix $\varepsilon>0$ and $h\in H$, we defined \textit{the difference quotient of $Z$ along $h$} denoted by $X_\varepsilon$ as
$$ X_{\varepsilon}:= \frac{Z\circ \tau_{\varepsilon h} - Z}{\varepsilon}.$$  
Our analysis will be based on notions of G\^ateaux Differentiability on the Wiener space. Before introducing a new definition in this context we recall the classical formulation due to Kusuoka and Stroock in \cite{KusuokaStroock}.  
\begin{definition}[Stochastically G\^ateaux Differentiability, \cite{KusuokaStroock}]
Let $p>1$ and  $Z$ be in $L^p(\R)$. $Z$ is said to be Stochastically G\^ateaux Differentiable if there exists $\mathcal{D}Z \in L^p(H)$ such that for any $h$ in $H$
$$ \lim\limits_{\varepsilon\to 0}  \frac{Z\circ \tau_{\varepsilon h}-Z}{\varepsilon}=\langle \mathcal{D}Z, h\rangle_H, \; \text{in probability under } \mathbb{P}_0 . $$
We refer this property as $(SGD)$.
\end{definition}
\noindent We now introduce a stronger formulation of $($SGD$)$ that we name \textit{Strong Stochastically G\^ateaux Differentiability}.
\begin{definition}[Strong Stochastically G\^ateaux Differentiability]
Let $p>1$ and  $Z$ be in $L^p(\R)$. $Z$ is said to be Strongly Stochastically G\^ateaux Differentiable of order $(p,q)$ if there exist $q\in (1,p)$ and $\mathcal{D}Z \in L^p(H)$ such that for any $h$ in $H$
$$ \lim\limits_{\varepsilon\to 0}  \E\left[ \left|\frac{Z\circ \tau_{\varepsilon h}-Z}{\varepsilon}-\langle \mathcal{D}Z, h\rangle_H\right|^q\right]=0. $$
We denote this property by $(SSGD_p(q))$ and we define $\mathcal{G}_p(q)$ the space of random variables $Z$ which satisfy $(SSGD_p(q))$ with G\^ateaux derivative $\mathcal{D}Z$ in $L^p(H)$.
\end{definition}

\noindent We name this property \textit{Strong Stochastically G\^ateaux Differentiability} as it strengthens the \textit{Stochastic G\^ateaux Differentiability} property introduced in \cite{KusuokaStroock} where the $L^q$-convergence above is replaced with convergence in probability. 

\begin{remark}
\label{rk:forall}
According to \cite[Theorem 4.1 and Lemma 4.2]{MPR}, $Z$ enjoys Property $(SSGD_p(q))$ for an element $q$ in $(1,p)$ if and only if there exists $\mathcal{D}Z \in L^p(H)$ such that for \underline{any} $q'\in (1,p)$ and for any $h$ in $H$, it holds that
$$ \lim\limits_{\varepsilon\to 0}  \E\left[ \left|\frac{Z\circ \tau_{\varepsilon h}-Z}{\varepsilon}-\langle \mathcal{D}Z, h\rangle_H\right|^{q'}\right]=0. $$ 
As a consequence, if there exists $q$ in $(1,p)$ such that $Z$ satisfies Property $(SSGD_p(q))$, then it also satisfies Property $(SSGD_p(q'))$ for any $q'\in(1,p)$.
\end{remark}
\noindent In view of Remark \ref{rk:forall}, we from now write $(SSGD_p)$ for $(SSGD_p(q))$ with $q$ in $(1,p)$. Similarly, we write $\mathcal{G}_p$ for $\mathcal{G}_p(q)$.  We now recall the main result of \cite{MPR}.
\begin{theorem}[\cite{MPR}, Theorem 4.1]\label{thm:characterizationD12} Let $Z$ be a real random variable. Then 
$$ Z\in \mathbb{D}^{1,p}\; \Longleftrightarrow  Z \text{ satisfies } (SSGD_p).$$ 
In other words, $\mathbb{D}^{1,p}=\mathcal{G}_p$.
\end{theorem}

\vspace{0.5em}
\no At this stage, we would like point out that from the definition of Property $($SSGD$_p)$ we do not consider the case where the difference quotient associated to an element in $\mathcal{G}_p$ converges in $L^p$. This induces a stronger property than $($SSGD$_p)$ defined below.

\begin{definition}
Let $p>1$ and  $Z$ be in $L^p(\R)$. We say that $Z$ satisfies $($SSGD$_p(p))$ if there exists $\mathcal{D}Z \in L^p(H)$ such that for any $h$ in $H$
$$ \lim\limits_{\varepsilon\to 0}  \E\left[ \left|\frac{Z\circ \tau_{\varepsilon h}-Z}{\varepsilon}-\langle \mathcal{D}Z, h\rangle_H\right|^p\right]=0. $$
We define $\mathcal{G}_p(p)$ the space of random variables $Z$ which satisfy $(SSGD_p(p))$ with G\^ateaux derivative $\mathcal{D}Z$ in $L^p(H)$.
\end{definition}

\noindent Obviously, in view of Theorem \ref{thm:characterizationD12}, Property $($SSGD$_p(p))$ is sufficient for a given random variable to belong to $\D^{1,p}$. This fact was somehow known in the literature (\cite{bogachev} or \cite{janson}) although the characterization of Theorem \ref{thm:characterizationD12} was not. As a consequence, it is a natural question to wonder whether or not Condition (SSGD$_2$(2)) is also necessary for an element $Z$ to belong to $\D^{1,2}$. In other words

\vspace{0.5em}
\begin{center}
Do the spaces $\mathbb{D}^{1,2}$ and $\mathcal{G}_2(2)$ coincide ?
\end{center}
Addressing this question is the main goal of the next section. 
\section{Main results}\label{mainresults}
Throughout this section, we fix $T=1$ and we set $\phi(x):=\frac{e^{-x^2/2}}{\sqrt{2\pi}}$ for any $x\in \mathbb{R}$. Let $a$ be in $\mathbb{R}$ and consider the following random variable
\begin{equation}\label{CE}
 Z:=f(W_1),
 \end{equation}
with 
$$f(x):= e^{\frac{x^2}{4}} x^{-a} (2\pi)^{\frac14} \mathbf{1}_{x\geq \sqrt{2a}} +g(x)\mathbf{1}_{x<\sqrt{2a}},$$
where $g$ is a bounded continuously differentiable map with bounded derivative, such that $f$ is continuously differentiable on $\mathbb{R}$. Notice that $f$ is a nonnegative and nondecreasing mapping when $x\geq \sqrt{2a}$.

\begin{theorem}
\label{th:main}
Let $p>1$. There exists $Z$ in $\mathbb{D}^{1,p}$ such that $Z$ does not satisfy $(\textrm{SSGD}_p(p))$.
\end{theorem}

\begin{proof} 
For the sake of simplicity we prove the result for $p=2$.

\vspace{0.5em} 
\noindent
Let $Z$ be the random variable defined by \eqref{CE} with $a>\frac32$. We prove that $Z\in \mathbb{D}^{1,2}$ \underline{and} that $Z$ does not satisfy (SSGD$_2$(2)). 
By definition of $f$,
$$ f'(x)=  \mathbf{1}_{x \geq \sqrt{2a}} e^{\frac{x^2}{4}} (-a x^{-a-1} +\frac12 x^{1-a}) (2\pi)^\frac14 +g'(x)\mathbf{1}_{x<\sqrt{2a}}. $$
The proof is divided in two steps. 

\paragraph*{Step 1. $\mathbf{Z\in \mathbb{D}^{1,2}}$.}
There exists $C>0$ such that
\begin{align*}
\mathbb{E}\left[ \left| Z\right|^2\right]& \leq 2\left( \int_{\sqrt{2a}}^{+\infty} x^{-2a} dx +\int_{-\infty}^{\sqrt{2a}} |g(x)|^2 \phi(x) dx \right) \\
& \leq C\left( \int_{\sqrt{2a}}^{+\infty} x^{-2a} dx+1\right)<+\infty, \; \textrm{since } a> \frac12.
\end{align*}
Besides, there exists $C'>0$ such that
\begin{align*}
\mathbb{E}\left[ \left| f'(W_1)\right|^2\right]& \leq 2\left(\int_{\sqrt{2a}}^{+\infty} (-a x^{-a-1} +\frac12 x^{1-a})^2 dx + \int_{-\infty}^{\sqrt{2a}} |g'(x)|^2 \phi(x) dx\right) \\
& = C'\left( \int_{\sqrt{2a}}^{+\infty} (x^{-2a-2} + x^{-2a+2} + x^{-2a})dx +1 \right)<+\infty, \; \textrm{since } a> \frac32.
\end{align*}
We hence deduce that $Z\in \mathbb{D}^{1,2}$. 

\paragraph*{Step 2. $\mathbf{Z}$ does not satisfy ($\mathbf{SSGD_2(2)}$).} Since $Z$ is in $\mathbb{D}^{1,2}$, we can compute its Malliavin derivative: 
$$\langle \nabla Z,h\rangle_H =\int_0^1 f'(W_1)\dot{h}_r dr=f'(W_1)h_1.$$
In addition, from \cite[Theorem 3]{sugita}, for any $h\in H$, it holds that
$$X_\varepsilon:= \frac{Z\circ \tau_{\varepsilon h} - Z}{\varepsilon} \overset{\mathbb{P}_0}{\underset{\varepsilon \to 0}{\longrightarrow}} \langle \nabla Z,h \rangle_H.$$
Thus, $Z$ belongs to $\mathcal{G}_2(2)$ if and only if $(|X_\varepsilon|^2)_{\varepsilon \in (0,1)}$ is uniformly integrable for every $h$ in $H$. In fact, in our example, for any $\varepsilon\in (0,1)$, $|X_\varepsilon|^2$ is even not integrable. Indeed, let $h\equiv 1$ in $H$ and fix $\varepsilon \in (0,1)$. Since $f$ is nondecreasing and nonnegative on $[\sqrt{2a},+\infty)$, we obtain
\begin{align*}
I&:= \mathbb{E}\left[|X_\varepsilon|^2 \right]= \int_\R \left|\frac{f(x+\varepsilon)-f(x)}{\varepsilon}\right|^2 \phi(x)dx\\
&\geq \int_{\sqrt{2a}}^{+\infty} \left|\frac{f(x+\varepsilon)-f(x)}{\varepsilon}\right|^2 \phi(x)dx\\
& \geq \int_{\sqrt{2a}}^{+\infty} \left|\frac{f(x+\varepsilon)-f(x+\varepsilon/2)}{\varepsilon}\right|^2 \phi(x)dx\\
& = \int_{\sqrt{2a}}^{+\infty} \left|\frac1\varepsilon\int_{\varepsilon/2}^\varepsilon f'(x+s)ds\right|^2 \phi(x)dx.
\end{align*} 
According to the mean value theorem, for any $x\in [\sqrt{2a},+\infty)$, there exists $s_\varepsilon^x\in (\varepsilon/2,\varepsilon)$ such that
$$Z_\varepsilon^x:= \varepsilon^{-1} \int_{\varepsilon/2}^\varepsilon f'(x+s)ds = f'(x+s_\varepsilon^x).$$
Thus, 
\begin{align*}
I\geq \int_{\sqrt{2a}}^{+\infty} \left| e^{\frac{(x+s_\varepsilon^x)^2}{4}}\psi_a(s_\varepsilon^x,x)\right|^2  e^{-\frac{x^2}{2}}dx= \int_{\sqrt{2a}}^{+\infty} \Gamma_{\varepsilon,x} \left|\psi_a(s_\varepsilon^x,x) \right
|^2  dx,\end{align*}
where $\psi_a(s,x):=-a (x+s)^{-a-1} +\frac12 (x+s)^{1-a}$ and $\Gamma_{\varepsilon,x}:=e^{\frac{(s_\varepsilon^x)^2}{2}+ s_\varepsilon^x x}.$

\vspace{0.5em}
\noindent Since $s_\varepsilon^x \in (\varepsilon/2,\varepsilon)$, we deduce that $I =+\infty$. Thus, for any $\varepsilon>0$, $|X_\varepsilon|^2$ is not integrable, so the family $(|X_\varepsilon|^2)_{\varepsilon \in (0,1)}$ is not uniformly integrable, and from the (converse part) of the dominated convergence Theorem, $Z\notin \mathcal{G}_2(2)$.
\end{proof}

\begin{remark} In the previous example we can recover the result of Theorem \ref{thm:characterizationD12} since $Z\in \mathbb{D}^{1,2}$, that is $(|X_\varepsilon|^q)_\varepsilon$ is indeed $($uniformly$)$ integrable when $q<2$, and so $Z\in \mathcal{G}_2(q)$. Assume that $1<q<q'<2$, one could show that for any $\varepsilon \in (0,1)$ and $x\in [\sqrt{2a},+\infty)$ there exists $s_\varepsilon^x \in (0,\varepsilon)$ such that
\begin{align*}
\sup\limits_{\varepsilon\in(0,1)} \mathbb{E}\left[|X_\varepsilon|^{q'} \right]=&\  \sup\limits_{\varepsilon\in (0,1)}\int_{\sqrt{2a}}^{+\infty} |f'(x+s_\varepsilon^x)|^{q'}\phi(x)dx+\sup\limits_{\varepsilon\in (0,1)}\int_{-\infty}^{\sqrt{2a}} |f'(x+s_\varepsilon^x)|^{q'} \phi(x)dx\\
=&\  \sup\limits_{\varepsilon\in (0,1)}\int_{\sqrt{2a}}^{+\infty} \left| e^{\frac{(x+s_\varepsilon^x)^2}{4}}\psi_a(s_\varepsilon^x,x)\right|^{q'}  \frac{e^{-\frac{x^2}{2}}}{(2\pi)^{\frac12-\frac{q'}{4}}}dx\\
&+ \sup\limits_{\varepsilon\in (0,1)}\int_{-\infty}^{\sqrt{2a}} |g'(x+s_\varepsilon^x)|^{q'}  \phi(x)dx\\
=&\  \sup\limits_{\varepsilon\in (0,1)}\int_{\sqrt{2a}}^{+\infty} \Gamma_{\varepsilon,x} \left|\psi_a(s_\varepsilon^x,x) \right
|^{q'} (2\pi)^{\frac{q'}{4}-\frac12} e^{\frac{x^2(\frac{q'}{2}-1)}{2}}dx\\
&+ \sup\limits_{\varepsilon\in (0,1)}\int_{-\infty}^{\sqrt{2a}} |g'(x+s_\varepsilon^x)|^{q'}  \phi(x)dx,
\end{align*}
where $\psi_a(s,x):=-a (x+s)^{-a-1} +\frac12 (x+s)^{1-a}$ and $\Gamma_{\varepsilon,x}:=e^{\frac{q' (s_\varepsilon^x)^2+2 q' s_\varepsilon^x x}{4}}.$

\vspace{0.5em}
\noindent Since the second term on the right-hand side above is uniformly bounded in $\varepsilon$ and since $q'\in (1,2)$
we deduce that $\sup\limits_{\varepsilon\in(0,1)} \mathbb{E}\left[|X_\varepsilon|^{q'} \right] <+\infty$. Thus, from the de la Vall\'ee-Poussin Criterion and the dominated convergence Theorem, $Z\in \mathcal{G}_2(q)$. 

\end{remark}
\noindent
We have thus constructed a random variable $Z$ in $\mathbb{D}^{1,2}$ which does not satisfy (SSGD$_2$(2)). Hence, according to Remark \ref{rk:forall}, for any $\varepsilon>0$, $Z\notin \mathbb{D}^{1,2+\varepsilon}$. Theorem \ref{th:main} is linked to another question, namely the right-continuity of the Malliavin-Sobolev spaces.
\begin{theorem}\label{thm:inclusion}
Let $p>1$. Define $ \mathbb{D}^{1,p+}:= \bigcup_{\varepsilon >0} \mathbb{D}^{1,p+\varepsilon} $. We have
\begin{equation}\label{inclu:d+dansd} \mathbb{D}^{1,p+}   \subsetneq   \mathcal{G}_p(p)   \subsetneq \mathbb{D}^{1,p}.\end{equation} \end{theorem}

\begin{proof}
Theorem \ref{th:main} gives immediately the last inclusion $ \mathcal{G}_p(p)  \subsetneq \mathbb{D}^{1,p}$. Now we turn to the first inclusion. Let $Z$ in $\mathbb{D}^{1,p+}$, then there exists $\varepsilon >0$ such that $Z\in\mathbb{D}^{1,p+\varepsilon}$. From Theorem \ref{thm:characterizationD12} together with Remark \ref{rk:forall}, we obtain that $Z \in \mathbb{D}^{1,p+\varepsilon}=\mathcal{G}_{p+\varepsilon}\subset \mathcal{G}_{p}(p)$. Hence, 
$$ \mathbb{D}^{1,p+}   \subset  \mathcal{G}_p(p). $$
However, $    \mathcal{G}_p(p)  \not\subset  \mathbb{D}^{1,p+} $. For the sake of simplicity, assume that $p=2$. We aim at constructing a random variable $Z$ which is in $\mathcal{G}_2(2)$ but not in $\mathbb{D}^{1,2+}$. Fix some $\eta,\mu$ such that $0<\eta <\mu<e^{-1} $ and such that
\begin{align}\label{choiceetamu}
& \bullet \; x\longmapsto 1/(x|\log(x)|^{8}) \textrm{ is decreasing on } (0, \mu+\eta).\\
&\bullet \label{choiceeta1} \; x\longmapsto \left|\frac{x}{\log(x)^6} \log\left(\frac{x}{\log(x)^6} \right)\right| =  -\frac{x}{\log(x)^6} \log\left(\frac{x}{\log(x)^6}\right) \textrm{ is increasing on } (0,\eta].\\
&\bullet \label{choiceeta2}  \; x\longmapsto x(\log(x))^{-6} \textrm{ is increasing on } (0,\eta].
\end{align}
Notice that \eqref{choiceetamu} implies that the map $x\mapsto 1/(x|\log(x)|^{i})$ is decreasing on $(0, \mu+\eta)$ for $i\in \{5,6,7,8\}$.
Set $Z:=f(W_1)$ where $f:\mathbb{R}\longrightarrow \mathbb{R}$ is defined by
\begin{equation}\label{def_f_inclusion} f(x):= F(x) \mathbf{1}_{x\in (0,\mu]} + G(x)\mathbf{1}_{x>\mu} ,\end{equation}
where $G$ is a smooth function on $\mathbb R_+^*$ with compact support such that $f$ is continuously differentiable on $\mathbb{R}_+^*$, and where $F:\mathbb{R}^*_+\longrightarrow \real$ is continuous and continuously differentiable on $(0,\mu]$ and defined by 
$$ F(x)=\frac{\sqrt{x}}{\log(x)^3}, \quad F'(x)=\frac{1}{2\sqrt{x} \log(x)^{3} }-\frac{1}{\sqrt{x}\log(x)^{4}}.$$

Then, using Bertrand's integrals Theory, one can easily prove that $Z\in \mathbb{D}^{1,2}$ but $Z\notin \mathbb{D}^{1,2+}$. Nevertheless, one can show that $Z\in \mathcal{G}_2(2)$. Indeed, from \cite[Theorem 3]{sugita}, for any $h\in H$
$$X_\varepsilon:= \frac{Z\circ \tau_{\varepsilon h} - Z}{\varepsilon} \overset{\mathbb{P}_0}{\underset{\varepsilon \to 0}{\longrightarrow}} \langle \nabla Z,h \rangle_H.$$
Hence, $Z$ belongs to $\mathcal{G}_2(2)$ if and only if $(|X_\varepsilon|^2)_{\varepsilon\in (0, \eta/\mu)}$ is uniformly integrable for every $h$ in $H$. According to the de la Vall\'ee-Poussin Criterion (see e.g. the remark immediately after Theorem 1.1 in \cite{hu_rosalsky}) if there exists a non-negative measurable map $\psi$ from $\mathbb{R}^+$ into $\mathbb{R}^+$ such that for any $h_1 \neq 0$ $$\lim\limits_{x\to +\infty} \frac{\psi(x)}{x}=+\infty, \quad \sup_{\varepsilon\in (0,\eta/|h_1|)} \mathbb{E}\left[ \psi(|X_\varepsilon|^2)\right]<+\infty,$$
then $(|X_\varepsilon|^2)_{\varepsilon \in (0,\eta/ |h_1|)}$ is uniformly integrable for all $h\in H$. Take $\psi(x):= x\left|\log(x)\right|\mathbf{1}_{x>0}$ defined on $\mathbb{R}^+$.

\vspace{0.5em}
\noindent Notice that if $h_1=0$ then $X_\varepsilon=0$ so, $(|X_\varepsilon|^2)_{\varepsilon \in (0,\eta/ |h_1|)}$ is uniformly integrable. We now distinguish two cases. 

\paragraph*{Case $\mathbf{h_1>0}.$}
 We have
\begin{align*}
I&:=\sup\limits_{\varepsilon\in(0,\eta/ h_1)} \mathbb{E}\left[\psi(|X_\varepsilon|^2)  \right]\\
&= I_{\mu}^{+\infty}+ I_{0}^{\mu}+ I_{-\infty}^{0},
\end{align*}
where 
\begin{align*}
I_{\mu}^{+\infty}&:=\sup\limits_{\varepsilon\in(0,\eta/h_1)}\int_{\mu}^{+\infty} \left|\varepsilon^{-1} \int_0^{\varepsilon h_1}  f'(x+s)ds\right|^2 \left|\log\left(\left|\varepsilon^{-1} \int_0^{\varepsilon h_1} f'(x+s)ds\right|^2\right)\right|\phi(x)dx\\
I_{0}^{\mu}&:=\sup\limits_{\varepsilon\in(0,\eta/h_1)}\int_{0}^{\mu} \left|\varepsilon^{-1} \int_0^{\varepsilon h_1}  f'(x+s)ds\right|^2 \left|\log\left(\left|\varepsilon^{-1} \int_0^{\varepsilon h_1}  f'(x+s)ds\right|^2\right)\right|\phi(x)dx\\
I_{-\infty}^{0}&:= \sup\limits_{\varepsilon\in(0,\eta/h_1)}\int_{-\varepsilon h_1}^{0} \left|\frac{f(x+\varepsilon h_1)-f(x)}{\varepsilon}\right|^2 \left|\log\left(\left|\frac{f(x+\varepsilon h_1)-f(x)}{\varepsilon}\right|^2\right)\right|\phi(x)dx.
\end{align*}
From the definition \eqref{def_f_inclusion} of $f$, we obtain
\begin{align*}
I_{\mu}^{+\infty}&=\sup\limits_{\varepsilon\in(0,\eta/h_1)}\int_{\mu}^{+\infty} \left|\varepsilon^{-1} \int_0^{\varepsilon h_1}  G'(x+s)ds\right|^2 \left|\log\left(\left|\varepsilon^{-1} \int_0^{\varepsilon h_1}  G'(x+s)ds\right|^2\right)\right|\phi(x)dx\\
I_{0}^{\mu}&=\sup\limits_{\varepsilon\in(0,\eta/h_1)}\int_{0}^{\mu} \left|\varepsilon^{-1} \int_0^{\varepsilon h_1}  f'(x+t)dt\right|^2 \left|\log\left(\left|\varepsilon^{-1} \int_0^{\varepsilon h_1}  f'(x+t)dt\right|^2\right)\right|\phi(x)dx\\
I_{-\infty}^{0}&= \sup\limits_{\varepsilon\in(0,\eta/h_1)}\int_{-\varepsilon h_1}^{0} \left|\frac{f(x+\varepsilon h_1)}{\varepsilon}\right|^2 \left|\log\left(\left|\frac{f(x+\varepsilon h_1)}{\varepsilon}\right|^2\right)\right|\phi(x)dx.
\end{align*}
According to the mean value theorem, for any $\varepsilon \in (0,\eta/h_1)$ and for any $x\in \mathbb R^+$, there exist $s_\varepsilon^x\in (0,\varepsilon h_1)$ and $t _\varepsilon^x\in (0,\varepsilon h_1)$ such that
$$\varepsilon^{-1} \int_0^{\varepsilon h_1}  G'(x+s)ds=G'(x+s_\varepsilon^x) \text{ and } \varepsilon^{-1} \int_0^{\varepsilon h_1} f'(x+t)dt= f'(x+t_\varepsilon^x).$$
Hence, from the definition of $G$ (we remind the reader that $G$ has compact support), we deduce that $I_{\mu}^{+\infty}<+\infty$. Now, we denote by $i(\varepsilon,\mu)$ the subset of $(0,\mu)$ defined by $i(\varepsilon,\mu):=\{x\in (0,\mu)| x+t_\varepsilon^x \leq \mu\}$. There exists a constant $C>0$ which may vary from line to line such that
\begin{align*}
I_{0}^{\mu}=&\ \sup\limits_{\varepsilon\in(0,\eta/h_1)}\int_{0}^{\mu} \left|f'(x+t_\varepsilon^x)\right|^2 \left|\log\left(\left|f'(x+t_\varepsilon^x)\right|^2\right)\right|\phi(x)dx\\
\leq&\ C\left(\sup\limits_{\varepsilon\in(0,\eta/h_1)}\int_{i(\varepsilon,\mu)} \frac{1}{(x+t_\varepsilon^x)| \log(x+t_\varepsilon^x)|^{6} } \left|\log\left(\left|\frac{\log(x+t_\varepsilon^x)-2}{2\sqrt{x+t_\varepsilon^x} \log(x+t_\varepsilon^x)^{4} }\right|^2\right)\right|\phi(x)dx\right.\\
&+ \sup\limits_{\varepsilon\in(0,\eta/h_1)}\int_{i(\varepsilon,\mu)} \frac{1}{(x+t_\varepsilon^x)| \log(x+t_\varepsilon^x)|^{8} } \left|\log\left(\left|\frac{\log(x+t_\varepsilon^x)-2}{2\sqrt{x+t_\varepsilon^x} \log(x+t_\varepsilon^x)^{4} }\right|^2\right)\right|\phi(x)dx\\
&+\left. \sup\limits_{\varepsilon\in(0,\eta/h_1)}\int_{(0,\mu) \setminus i(\varepsilon,\mu)}  \left|G'(x+t_\varepsilon^x)\right|^2 \left|\log\left(\left|G'(x+t_\varepsilon^x)\right|^2\right)\right|\phi(x)dx\right)\\
=&\ C\left( \sup\limits_{\varepsilon\in(0,\eta/h_1)}\int_{i(\varepsilon,\mu)} \frac{1}{(x+t_\varepsilon^x)| \log(x+t_\varepsilon^x)|^{6} } \log\left(\frac{|\log(x+t_\varepsilon^x)-2|^2}{4(x+t_\varepsilon^x) \log(x+t_\varepsilon^x)^{8} }\right)\phi(x)dx\right.\\
&+ \left.\sup\limits_{\varepsilon\in(0,\eta/h_1)}\int_{i(\varepsilon,\mu)} \frac{1}{(x+t_\varepsilon^x)| \log(x+t_\varepsilon^x)|^{8} } \log\left(\frac{|\log(x+t_\varepsilon^x)-2|^2}{4(x+t_\varepsilon^x) \log(x+t_\varepsilon^x)^{8} }\right)\phi(x)dx+1\right).
\end{align*}
Since $\mu$ is smaller than $e^{-1}$, we can ensure that for any $\varepsilon\in (0,\eta/h_1)$ and $x\in i(\varepsilon,\mu)$, we have
 $$\left|\log(\left|\log(x+t_\varepsilon^x)\right|)\right|\leq |\log(x+t_\varepsilon^x)|.$$ Hence, there exists $C>0$ such that $$ I_1\leq C \sup\limits_{\varepsilon\in(0,\eta/h_1)}\sum_{i=5}^{8} \int_{i(\varepsilon,\mu)} \frac{1}{(x+t_\varepsilon^x)| \log(x+t_\varepsilon^x)|^{i} } dx.$$
It is then clear, from the properties of Bertrand's integrals together with the choice of $\eta,\mu$ according to \eqref{choiceetamu}, that $I_0^{\mu}<+\infty.$
Besides, 
\begin{align*}
I_{-\infty}^{0}&= \sup\limits_{\varepsilon\in(0,\eta/h_1)}\int_{-\varepsilon h_1}^{0} \left|\frac{f(x+\varepsilon h_1)}{\varepsilon}\right|^2 \left|\log\left(\left|\frac{f(x+\varepsilon h_1)}{\varepsilon}\right|^2\right)\right|\phi(x)dx\\
&=\sup\limits_{\varepsilon\in(0,\eta / h_1)}\frac{1}{\varepsilon^2}\int_{0}^{\varepsilon h_1} \frac{x}{\log(x)^6} \left|\log\left(\frac{x}{\varepsilon^2 \log(x)^6}\right)\right|\phi(x-\varepsilon h_1)dx\\
&=I_1+I_2,
\end{align*}
where 
\begin{align*}
I_1&=\sup\limits_{\varepsilon\in(0,\eta/h_1)}\frac{1}{\varepsilon^2} \int_{0}^{\varepsilon h_1} \frac{x}{\log(x)^6} \left|\log(\varepsilon^{-2})\right|\phi(x-\varepsilon h_1)dx\\
I_2&=\sup\limits_{\varepsilon\in(0,\eta/ h_1)}\frac{1}{\varepsilon^2}\int_{0}^{\varepsilon h_1} \left| \frac{x}{\log(x)^6} \log\left(\frac{x}{\log(x)^6} \right)\right|\phi(x-\varepsilon h_1)dx.
\end{align*}
We first show that $I_1$ is finite. From \eqref{choiceeta2} and since $\eta < e^{-1}<1$ we obtain
\begin{align*}
I_1&\leq \sup\limits_{\varepsilon\in(0,\eta/h_1)}  |h_1|^2\frac{|\log(\varepsilon^{-2})|}{\log(\varepsilon h_1)^6}\phi(0)\leq |h_1|^2 \phi(0)\left( -\frac{2}{\log(\eta)^5} + \frac{|\log(h_1)|}{\log(\eta)^6} \right)<+\infty.
\end{align*}
We now turn to $I_2$ and we obtain from \eqref{choiceeta1} together with  \eqref{choiceetamu} that
\begin{align*}
I_2\leq \sup\limits_{\varepsilon\in(0,\eta/h_1)} \frac{|h_1|^2}{\log(\varepsilon h_1)^6} \left|\log\left(\frac{\varepsilon h_1}{\log(\varepsilon h_1)^6}\right) \right|\phi(0)\leq -\frac{7 |h_1|^2}{\log(\eta)^5}\phi(0) <+\infty,
\end{align*}
which implies in turn $ I_{-\infty}^{0}<+\infty$.

\paragraph*{Case $\mathbf{h_1}<0$.} Similarly to the previous case, we have
\begin{align*}
J&:=\sup\limits_{\varepsilon\in(0,\eta/|h_1|)} \mathbb{E}\left[\psi(|X_\varepsilon|^2)  \right]\\
&= J_{+\infty}+ J_{\mu}+ J_{0},
\end{align*}
where 
\begin{align*}
J_{+\infty}&:=\sup\limits_{\varepsilon\in(0,\eta /|h_1|)}\int_{\mu+\varepsilon |h_1|}^{+\infty} \left|\varepsilon^{-1} \int_0^{\varepsilon |h_1|} f'(x-s)ds\right|^2 \left|\log\left(\left|\varepsilon^{-1} \int_0^{\varepsilon |h_1|} f'(x-s)ds\right|^2\right)\right|\phi(x)dx\\
J_{\mu}&:=\sup\limits_{\varepsilon\in(0,\eta/ |h_1|)}\int_{\varepsilon |h_1|}^{\mu+\varepsilon |h_1|} \left|\varepsilon^{-1} \int_0^{\varepsilon |h_1|} f'(x-s)ds\right|^2 \left|\log\left(\left|\varepsilon^{-1} \int_0^{\varepsilon |h_1|} f'(x-s)ds\right|^2\right)\right|\phi(x)dx\\
J_{0}&:= \sup\limits_{\varepsilon\in(0,\eta / {\varepsilon |h_1|})}\int_{0}^{\varepsilon |h_1|} \left|\frac{f(x-{\varepsilon |h_1|})-f(x)}{\varepsilon}\right|^2 \left|\log\left(\left|\frac{f(x-{\varepsilon |h_1|})-f(x)}{\varepsilon}\right|^2\right)\right|\phi(x)dx.
\end{align*}
From the definition \eqref{def_f_inclusion} of $f$, we obtain
\begin{align*}
J_{+\infty}&=\sup\limits_{\varepsilon\in(0,\eta/ |h_1|)}\int_{\mu+\varepsilon |h_1|}^{+\infty} \left|\varepsilon^{-1} \int_0^{\varepsilon |h_1|} G'(x-s)ds\right|^2 \left|\log\left(\left|\varepsilon^{-1} \int_0^{\varepsilon |h_1|} G'(x-s)ds\right|^2\right)\right|\phi(x)dx\\
J_{\mu}&=\sup\limits_{\varepsilon\in(0,\eta/|h_1|)}\int_{\varepsilon |h_1|}^{\mu+\varepsilon |h_1|} \left|\varepsilon^{-1} \int_0^{\varepsilon |h_1|} f'(x-t)dt\right|^2 \left|\log\left(\left|\varepsilon^{-1} \int_0^{\varepsilon |h_1|} f'(x-t)dt\right|^2\right)\right|\phi(x)dx\\
J_{0}&= \sup\limits_{\varepsilon\in(0,\eta/|h_1|)}\int_{0}^{\varepsilon |h_1|}\left|\frac{f(x)}{\varepsilon}\right|^2 \left|\log\left(\left|\frac{f(x)}{\varepsilon}\right|^2\right)\right|\phi(x)dx.
\end{align*}
From the mean value Theorem, one shows that $J_{+\infty}<+\infty$. Besides, by making the substitution $y=x-\varepsilon |h_1|$ in $J_{0}$ we deduce that $J_{0}\leq e^{\frac{|\eta|^2}{2}} I_{-\infty}^0<+\infty$. Now, turn to $J_{\mu}$. According to the mean value theorem, for any $\varepsilon \in (0,\eta/|h_1|)$ and for any $x\in \mathbb (\varepsilon |h_1|, \mu + \varepsilon |h_1|]$, there exists $t _\varepsilon^x\in (0,\varepsilon |h_1|)$ such that
$$\varepsilon^{-1} \int_0^{\varepsilon |h_1|} f'(x-t)dt= f'(x-t_\varepsilon^x).$$ So, following the same lines as in the proof of $I_{0}^{\mu}<+\infty$, we obtain $J_{\mu}<+\infty$. Indeed,
\begin{align*}
J_{\mu}=&\ \sup\limits_{\varepsilon\in(0,\eta/|h_1|)}\int_{\varepsilon |h_1|}^{\mu+\varepsilon |h_1|} \left| f'(x-t_\varepsilon^x)\right|^2 \left|\log\left(\left| f'(x-t_\varepsilon^x)\right|^2\right)\right|\phi(x)dx\\
\leq&\ \sup\limits_{\varepsilon\in(0,\eta/|h_1|)}\int_{j(\varepsilon,\mu)} \left| f'(x-t_\varepsilon^x)\right|^2 \left|\log\left(\left| f'(x-t_\varepsilon^x)\right|^2\right)\right|\phi(x)dx \\
&+ \sup\limits_{\varepsilon\in(0,\eta/|h_1|)}\int_{(\varepsilon|h_1|,\mu+\varepsilon|h_1|)\setminus j(\varepsilon,\mu)} \left| f'(x-t_\varepsilon^x)\right|^2 \left|\log\left(\left| f'(x-t_\varepsilon^x)\right|^2\right)\right|\phi(x)dx,
\end{align*} 
with $ j(\varepsilon,\mu):=\{ x\in (\varepsilon |h_1|, \mu+\varepsilon |h_1|), \; x-t_\varepsilon^x \leq \mu\}.$ Hence,
\begin{align*}
J_\mu\leq&\ C\left( \sup\limits_{\varepsilon\in(0,\eta/|h_1|)}\int_{j(\varepsilon,\mu)} \frac{1}{(x-t_\varepsilon^x)| \log(x-t_\varepsilon^x)|^{6} } \log\left(\frac{|\log(x-t_\varepsilon^x)-2|^2}{4(x-t_\varepsilon^x) \log(x-t_\varepsilon^x)^{8} }\right)\phi(x)dx\right.\\
&+ \left.\sup\limits_{\varepsilon\in(0,\eta/|h_1|)}\int_{j(\varepsilon,\mu)} \frac{1}{(x-t_\varepsilon^x)| \log(x-t_\varepsilon^x)|^{8} } \log\left(\frac{|\log(x-t_\varepsilon^x)-2|^2}{4(x-t_\varepsilon^x) \log(x-t_\varepsilon^x)^{8} }\right)\phi(x)dx+1\right)\\
\leq&\ C \left(1+ \sup\limits_{\varepsilon\in(0,\eta/h_1)}\sum_{i=5}^{8} \int_{j(\varepsilon,\mu)} \frac{1}{(x-t_\varepsilon^x)| \log(x-t_\varepsilon^x)|^{i} } dx\right).
\end{align*}
In addition, \eqref{choiceetamu} implies that
\begin{align*}
J_\mu&\leq C \left(1+ \sup\limits_{\varepsilon\in(0,\eta/h_1)}\sum_{i=5}^{8} \int_{j(\varepsilon,\mu)} \frac{1}{(x-\varepsilon |h_1|)| \log(x-\varepsilon |h_1|)|^{i} }dx \right)\\
&\leq C \left(1+ \sup\limits_{\varepsilon\in(0,\eta/h_1)}\sum_{i=5}^{8} \int_{\varepsilon |h_1|}^{\mu+\varepsilon |h_1|} \frac{1}{(x-\varepsilon |h_1|)| \log(x-\varepsilon |h_1|)|^{i} }dx \right)\\
&=C \left(1+ \sum_{i=5}^{8} \int_{0}^{\mu} \frac{1}{y| \log(y)|^{i} }dy \right)<+\infty.
\end{align*}

\noindent This concludes the proof. 

\end{proof}

\begin{remark}
Theorem \ref{thm:inclusion} describes the fine structure of $\mathbb{D}^{1,p}$. However, giving an explicit characterization of $\mathbb{D}^{1,p+}$ seems to be a far more complicated result. Indeed, even for the $L^p$-spaces, it is quite difficult to characterize $\bigcup_{\varepsilon >0}L^{p+\varepsilon}.$ This problem was studied recently in \cite{KMM}, using Lorenz spaces Theory.
\end{remark}

\section*{Acknowledgments}

Thibaut Mastrolia is grateful to R\'egion \^Ile-de-France for financial support and acknowledges INSA de Toulouse for its warm hospitality.

\end{document}